\documentclass[12pt]{article}

\usepackage{amsmath,amsthm,amssymb}

\theoremstyle{plain}
\newtheorem{corollary}{Corollary}
\newtheorem{lemma}{Lemma}%[section]
\theoremstyle{remark}
\newtheorem{rem}{\indent \sc Remark}

\newcommand{\de}{\delta}
\newcommand{\ep}{\varepsilon}
\newcommand{\lam}{\lambda}
\newcommand{\om}{\omega}

\newcommand{\tht}{\theta}
\newcommand{\Ga}{\Gamma}
\newcommand{\beq}{\begin{eqnarray*}}
\newcommand{\eeq}{\end{eqnarray*}}
\newcommand{\beqn}{\begin{equation}}
\newcommand{\eeqn}{\end{equation}}
\newcommand{\nin}{\noindent}
\newcommand{\pf}{\noindent {\it Proof. \,}}
\newcommand{\R}{\mathbb{R}}

\def\v2{\vskip2mm}

\begin{document}

\begin{center}
{\Large A  renewal theorem for relatively stable variables }
\vskip6mm
{K\^ohei UCHIYAMA\footnote{Department of Mathematics,Tokyo Institute of Technology, Japan}} \\
\vskip2mm
%{Department of Mathematics,\\ Tokyo Institute of Technology} \\
%{Oh-okayama, Meguro Tokyo 152-8551\\
%e-mail: \,ko-hei-3@math.titech.ac.jp}
\end{center}
\v2
\vskip6mm

\begin{abstract} 
Let $F\{dx\}$ be a relatively stable probability distribution on the whole real line and $S_n$ the random walk started at the origin with step distribution $F$. We obtain an exact asymptotic form of the Green measure  $U\{x+dy\}= \sum_{n=0}^\infty P[S_n-x \in dy]$  as $x\to \infty$ when  $S_n$ is transient and  $S_n\to \infty$ in probability.    
If $F$ is concentrated on $[0,\infty)$, it is relatively stable if and only if  $\ell(x) :=\int_0^x F\{(t,\infty)\}dt$  is  slowly varying at infinity;  our result entails that if   $F$ is non-arithmetic and relatively stable, then 
 $\lim_{x\to\infty}\, \ell(x)U\{[x, x+h)\} = h$ for each  $h>0$.  
 This 
  surpasses 
  the known  result due to Erickson \cite{Ec},  the latter assuming the stronger condition that $xF\{(x,\infty)\}$ is  slowly varying. 
  An obvious analogue also holds for  arithmetic variables. 
   \end{abstract}



\vskip6mm
\section{Introduction and results}

Let $X$ be a real-valued random variable  with distribution function $F(x)$, $x\in \R$ and $F\{dx\}$ the probability measure of $X$.  
Let $U$ be the associated  renewal measure (or, what is the same,  the Green measure of the random walk generated by $F$) that is given by 
$$U\{I\}= \sum_{n=0}^\infty F^{n*}\{I\} = \sum_{n=0}^\infty P[S_n \in I],$$
for finite intervals  $I\subset \R$.  Here $F^{n*}$ denotes the $n$-fold convolution of $F$ and $S_n =X_1+\cdots +X_n$ with $X_j$, $j=1, 2, \ldots$ independent copies  of $X$ and $S_0=0$.  The above sum may be infinite. Following \cite[Section VI.10]{F}  we call $F$ {\it transient} if $U\{[-y,y]\}<\infty$ for all $y>0$
 (otherwise  $U\{[-y,y]\}=\infty$ for all $y>0$).   
 Since the 1940s the asymptotic form of  $U[x,x+h):= U\{[x, x+h)\}$   has been studied in an intensive series of works,  first, for renewal processes (namely,  when $X$ assumes non-negative values only) and then,  for $X$ having a positive mean and taking both positive and negative values. For the renewal processes the now well-known classical results are  obtained  by Erd\"os-Feller-Pollard  \cite{EFP} in 1949 for lattice random variables and  immediately  afterward by Blackwell \cite{Bl} for non-lattice ones. 
After many attempts to make extensions to general distributions in  various special classes  Feller and Orey \cite{FO} proved at   last  in 1961 a general theorem that covered  the results got up to that time. It states that
   if  $F$ is transient and  non-arithmetic,  
 then  for each $h>0$, 
 \beqn\label{1}
 \lim_{x\to\infty} U[x, x+h) = \left\{ \begin{array} {ll}   h/EX \quad &\mbox{if \quad  $E|X|<\infty$ and $EX>0$},\\
 0  \quad &\mbox{otherwise},    
 \end{array}
\right.
\eeqn 
for  arithmetic walks the  obvious analogue being shown by virtually the same proof.  [In   
 \cite{FO} Fourier analytic method  is employed;  a simpler and more elementary proof of (\ref{1}) is given  in  \cite[Section XI.9]{F}.]
The purpose of the present paper is   to find    exact asymptotic forms  of $U[x,x+h)$ as $x\to\infty$ for the  relatively stable random variables with $E|X|=\infty$, which constitute a significant  class of probability laws.

For  non-negative   $X$ with $EX=\infty$  the asymptotic behaviour of $U[x,x+h)$   has been studied when  $F$ is in the domain of attraction of a stable law  \cite{GL}, \cite{Ec}, \cite{CD} etc.
It is known (obtained as  a consequence of  Karamata's Tauberian theorems as given in \cite{F}, \cite{BGT}) that
$$\ell(t):=x^{\alpha -1}\int_0^x [1-F(t)]dt \quad \mbox{is slowly varying at infinity} 
$$
if and only if 
\beqn\label{Int_RWT}
U[0,x) \sim
\kappa_\alpha x^\alpha/\ell(x) \qquad \big(\kappa_\alpha =1\big/\big[\Ga(2-\alpha)\Ga(1+\alpha)\big]\big),
 \eeqn
   where $F(x) =F(-\infty,x]$ and $\sim$ means that the ratio of its two sides converges to unity.  (A real function $f(x)$ is said to be {\it slowly varying  (s.v.)}  at infinity if $f(x)>0$ for  $x$ large enough and  $\lim_{x\to\infty}f(\lam x)/f(x) =1$ for any $\lam>1$.)
The local versions of this  have  been dealt  with.   Let $F$ be non-arithmetic for simplicity. It is shown by Erickson \cite{Ec} (cf. also  \cite{GL},  Section 8.6 \cite{BGT}) that if $\ell$ above is s.v. and $0<\alpha<1$,
 then 
 \beqn\label{Linf}
\begin{array}{ll}
(a) \quad \lim U[x,x+h)\big/\big[U[0, x)/x] = \alpha h \quad &\mbox{for}\quad \alpha >1/2,\\[1mm]
(b) \quad \liminf U[x,x+h)\big/\big[U[0, x)/x] = \alpha h \quad &\mbox{for}\quad \alpha \leq 1/2,
\end{array}
\eeqn
where $\liminf $ cannot be replaced by  $\lim$ in general \cite{W}, thus  the local version of (\ref{Int_RWT}) is true if $1/2<\alpha<1$ but not for $\alpha\leq 1/2$. 
 For $\alpha=1$   the above formula  (\ref{Linf}a) is also obtained  by Erickson  \cite{Ec} but 
under the condition that for  some s.v. function  $L$
$$
(*)\quad 1- F(x) \sim L(x)/x, \quad
$$
which is stronger than the slow variation of $\ell$. [Note that for $0<\alpha<1$ the slow variation of $\ell$ is equivalent to $1-F(x) \sim (1-\alpha)x^{-\alpha}\ell(x)$.]
In the recent article \cite{CD},  Caravenna-Doney  obtained detailed estimates of  $U[x,x+h)$  for  random walks and L\'evy processes that are in the domain of attraction of a stable law with index $0<\alpha<1$, especially they derived a necessary and sufficient condition given in terms of $F$ in order  for  the local version of (\ref{Int_RWT}) to be true. They also studied the problem when  $X$ assumes   both positive and negative values and  extended their one-sided result to that case.
When $F$ is concentrated on the integer lattice $\mathbb{Z}$, is  transient and belongs to the Cauchy domain of attraction  Berger \cite[Theorem 3.6 ]{Ber} obtained an asymptotic form of the  function $U\{n\}$  under some additional assumption on $P[X=n]$.

The question when the local version of (\ref{Int_RWT}) holds seems  to have been well addressed
in most cases, but there is still  an issue that remains  to be worked out:  
  To ensure  (\ref{Linf}) for $\alpha=1$  can one weaken condition ($*$)? What about the case $X$ takes vales of both signs?
In this article, we consider a transient random walk on the whole real line $\mathbb{R}$  
whose  step distribution is relatively stable. 
 The random variable  $X$ (or its distribution  $F$)  is said to be {\it relatively stable} ({\it r.s.})  if there exist  real norming  constants  $\lam_n$ such that
 $S_n/\lam_n$ converges to unity in probability;
in this case   $\lam_n$ are regularly varying with index 1 and eventually  positive or eventually negative  \cite{R76} and after \cite{KM0} we say that $X$ is {\it positively r.s.}  in the former case   and {\it negatively r.s.} in the latter.  In the sequel we shall always  suppose that
\beqn\label{EX}
E|X|=\infty
\eeqn
and  find the exact asymptotic form of $ U[x, x+h)$ as $x\to\infty$ when  $F$ is positively r.s. as well as  transient,
 the result entailing that  $U[x,x+h)$ is s.v. at infinity; 
  in particular  we shall improve on  the above mentioned results of  \cite{Ec} and \cite{Ber}  by removing the extra conditions on $F$  (see Corollary \ref{cor2} for the former one and Section 4.3 for the latter).  For one-dimensional  random walks attracted to a stable law of exponent  $1\leq \alpha\leq 2$, 
  the situation where  the ascending ladder height variable, say $Z$, is r.s. naturally arises.  
    Since  the relative stability of $Z$ is equivalent to the slow variation of $\int_0^xP[Z>t]$,  Erickson's result does not apply to  the corresponding renewal measure without assuming some extra condition that ensures $(*)$ for the distribution of $Z$, but it is not easy to  specify an appropriate one in terms  of the step distribution of the walk.  The author encountered  such a situation in studying the two-sided exit problem  of random walks   \cite{Uexit} and got an interest in the present study.

\v2
  
 To state the result, we introduce notation. 
Put for $x\geq 0$, 
$$H(x) =1-F(x) + F(-x-0), \quad K(x)= 1-F(x) - F(-x-0), $$
$$\ell(x) = \int_0^x H(t)dt \quad\mbox{and }\quad A(x) = \int_0^x K(t) dt.$$
Since $E|X|=\infty$,  we have $H(x)>0$  for all  $x$.   We require the following conditions:
 \beqn\label{A_sv}
 \left\{
 \begin{array}{ll}
{\rm (Ha)}\;\quad A(x)/xH(x) \to \infty \quad (x\to\infty).\\[2mm]
{\rm (Hb)}\quad  \int_{x_0}^\infty H(x)dx/A^2(x) <\infty\quad \mbox{for some  $x_0>0$.} 
\end{array} \right. \qquad
\eeqn
Condition (Ha) is equivalent to positive relative stability of the walk \cite{R76}, \cite{M}, and (Hb) is equivalent to transience of it under (Ha) as being ensured by the  theorem below.  A simple example that satisfies   (Hab)---conjunction of (Ha) and (Hb)---is provided by any distribution such that  $F(-x) \sim (1-p)L(x)/x$ and $1-F(x)  \sim pL(x)/x$ as $x\to\infty$ with $1/2 < p\leq 1$ and  $L$ a s.v.\,function satisfying  $\int_1^\infty L(t) dt/t=\infty$
(see Section 4.3  where  the case $p=1/2$ is also discussed). One can easily deduce   that both $A$ and $\ell$ are  s.v. under (Hab) (see Remark \ref{rem1} below). The condition (Hb)   entails that the derivative $(1/A)'(x) = -K(x)/A^2(x)$ is integrable about $x=\infty$ and $\limsup A(x)=\infty$. Hence  
\beqn\label{1/A}
\int_x^\infty \frac{K(t)}{A^2(t)}dt = \frac{1}{A(x)}  \quad  (x>x_0) \quad \mbox{under \; (Hb)},
\eeqn
entailing  $A(x)\to\infty$.

We call  $X$  (or $F$) {\it arithmetic} if $X/h_0\in \mathbb{Z}$ almost surely (a.s.) for some $h_0>0$ and {\it non-arithmetic} otherwise. We shall suppose $F$ is non-arithmetic,   the arithmetic case being similarly dealt with in a simpler way. Put
\[
r_+(x) = \int_x^\infty\frac{1-F(t)}{ A^2(t)}dt   \quad \mbox{ and} \quad r_-(x)=  
 \int_{|x|}^\infty\frac{F(-t)}{ A^2(t)}dt.
 \]  
\v2
%THM
\noindent
{\bf Theorem.}  \, {\it Suppose that   (Ha)  is satisfied.
 Then $F$ is transient  if and only if (Hb) holds, and   if this is the case and  $F$ is non-arithmetic, then   for each  $h>0$,  
 \beqn\label{eq_T2}
  \frac{U(x, x+h]} {h} =
  \left\{ \begin{array}{ll} r_+(x) \{1+o(1)\}  \quad & \mbox{as $x\to +\infty$,}\\[1mm]
r_-(x) +o\big(r_+(|x|)\big)  \quad & \mbox{as $x\to -\infty$.}
 \end{array} \right.   
 \eeqn}

\v2

By (\ref{1/A})  $ r_+(x) =r_-(x)+1/A(x)\geq  r_-(x)\vee[1/A(x)$, $x>x_0$. 
 If $F(-x)/[1-F(x)]$ is bounded away from zero, then  the error term $o(r_-(|x|))$ is really negligible  in (\ref{eq_T2}), so that   ${U(x, x+h]}$ and ${U(-x, -x+h]}$ are comparable as $x\to\infty$ in spite of the fact that for $n$ large enough  $S_n$ is located around $\lam_n \sim nA(\lam_n)$ with overwhelming probability (under (Ha))---although $\liminf S_n=-\infty$ a.s. (see Remark \ref{rem6}). 
   If $F(-x)/[1-F(x)]$ tends to zero (or less rigidly  $r_-(x)/r_+(x) \to 0$; see Remark \ref{rem6}), then  $A(x)\sim \ell(x)$ and  the second case of (\ref{eq_T2}) comes down to merely  $U(x,x+h] =  o(1/\ell(x))$  ($x\to-\infty$) while $U(x,x+h] = h/\ell(x)$ as $x\to +\infty$.  Note that $r_+$ is s.v.\,whenever (Hab) holds, since then $x[1-F(x)]/ [A^2(x)r_+(x)] \leq xH(x)/A(x)\to 0$.  
   
  If $F$ is arithmetic of span 1 and recurrent,  one has the potential function $a(x) =\sum_{n=0}^\infty(P[S_n=0] - P[S_n=-x])$ that is well defined for every integer $x$ and plays a crucial role in the potential theory of random walk  [20, Section 28].  If such an $F$ satisfies (Ha) ($E|X|$ may be finite), then $a(x)$ admits  an asymptotic estimate analogous to (7) (see [22, Theorem 7]). 
 
\v2
  
Let (Hab) hold and define a function  $m(x)$ via
\[
\frac1{m(x)} = \int_{x}^\infty \frac{H(t)dt}{A^2(t)}, \qquad x>x_0.
\]
 Comparing this integral  with the corresponding one for $1/A(x)$ given in  (\ref{1/A})  one obtains
  \beqn\label{m/A/ell}
m(x)\leq A(x)\leq \ell(x).
\eeqn  
It is also noted that 
\beqn\label{1/ell}
 \frac1{\ell(x)} =\int_x^\infty \frac{H(t)}{\ell^2(t)}dt.
\eeqn

Formula (\ref{eq_T2})  is simplified under  the following specific condition
$${\rm (Hc)}\;\quad  \;\; \kappa := \lim_{x\to\infty} m(x)/A(x)  
$$
with the understanding that it entails (Hb). For $\kappa>0$, (Hc) is equivalent to $\lim A(x)/\ell(x)=\kappa$ under (Hab) (see Remark \ref{rem4}).
By  (\ref{1/A})  and (\ref{1/ell}) $r_+(x)-r_-(x) = 1/A(x)$ and $r_+(x) + r_-(x) =1/m(x)$. By these    identities  our  Theorem    reduces to   the following result. 
\v2
\nin
%Cor1
\begin{corollary} \label{cor1} 
  If   (Hac)   holds, 
  then $F$ is transient and  for each  $h>0$, 
 $$ m(|x|)\frac{U(x, x+h]}{h} \, \longrightarrow \,
  \left\{ \begin{array}{ll}
\frac12 (1 + \kappa)  \quad &\mbox{as}\;\;  x \to + \infty, 
   \\[1mm]
  \frac12 (1 - \kappa)  \quad &\mbox{as}\;\;  x \to  -\infty.
 \end{array}\right.
  $$
\end{corollary}

If $X$ assumes only non-negative values so that
$\ell(x)=\int_0^x(1-F(t))dt$ agrees with both $m(x)$ and $A(x)$, Corollary \ref{cor1}  further
reduces to
 
%Cor2
\begin{corollary} \label{cor2}  If $\ell$ is s.v. and $X$ is non-negative a.s. and non-arithmetic, then  it holds that
 \beqn\label{m_res}
\lim_{x\to\infty} \, \ell(x)U[x, x+h)  = h  \quad \mbox{for each} \quad h>0. 
\eeqn
 \end{corollary}
 
When $F$ is  arithmetic,  the analogous result---i.e., (\ref{m_res}) holds if  $h$ is a multiple of the span of $F$---is verified in \cite[Appendix(B)]{Upot} with  a relatively simpler proof.  

\v2
%REM1 
\begin{rem}\label{rem1}    Condition (Ha) 
 implies 
  $$({\rm Ha}')\;\;\quad  \exists x_0>0, \;\; A(x) \; \mbox{is positive for  $x\geq x_0$  and  s.v.\,as $x \to +\infty$.}$$
 In fact $\log [A(x)/A(x_0)] = \int_{x_0}^x \ep_A(t)dt/t$ with $\ep_A(t):= tK(t)/A(t)$ so that $A$ is a (normalized) s.v.\!\! function under (Ha). [$m$  admits a  similar relation with $\ep_m(t)= tm(t)H(t)/A^2(t)$ and similarly for $\ell$.]  
 If  (Hc) holds with $\kappa>0$, then  the converse is true, i.e.,  (Ha) follows from 
  (Ha$'$c), since (Hc) implies  $A(x)/\ell(x)\to \kappa$ (see Remark \ref{rem4} below) and then the slow variation of $A(x)$ implies (Ha). Thus    (Hac) is equivalent to (Ha$'$c) in the case $\kappa >0$.   
 \end{rem}
 
 %REM2
 \begin{rem}\label{rem2}  The conditions (Ha) and (Hb) are not comparable, i.e., there exist two examples of $F$, one satisfying  (Ha) but violating  (Hb) and the other  satisfying  (Hb) but violating (Ha) (see Section 4.3). 
 \end{rem}
  
 %REM3
\begin{rem} \label{rem4}  Let  (Hab) hold.  If  $A(x)/\ell(x)$ approaches a positive constant as $x\to\infty$, then   (Hc) holds (necessarily with $\kappa =\lim A(x)/\ell(x)$). Converse is true if $\kappa>0$,  as is readily  verified by the identity
\[\ell(x) + \frac{A^2(x)}{m(x)} =2\int_{x_0}^x \frac{A(t)}{m(t)}K(t)dt + C.
\]
\end{rem}

%\REM4
\begin{rem}\label{rem3}  Suppose (Hab) to hold and consider the condition
 \[ 
{\rm (S)} \quad \exists C <\infty, \;\;  m(x)\ell(x) < CA^2(x) \quad \mbox{for all sufficiently large $x$}.
\] 
If (S) is valid, then   $ A(x)/\ell(x) \to  0$ implies (Hc)  with $\kappa =0$     so that 
$U(x, x+h]/h \sim 1/2m(x) $ as  $x\to \pm \infty$. Under some mild  regularity  conditions (S) holds.
Schwarz' inequality yields
 \[
 \frac1{m(x)\ell(x)} \geq \bigg[\int_x^\infty  \frac{H(t)}{\ell(t) A(t)}dt\bigg]^2,
 \]
 which shows that  (S) is satisfied 
 if  $\ell(x)/A(x)$ is almost increasing (i.e.,   for some constants $C$ and $x_0$, 
   $C \ell(y)/A(y) \geq \ell(x)/A(x)$   if $y>x >x_0$). Indeed,  using (\ref{1/ell}) the integral in the square brackets is  bounded from  below by $\inf_{t\geq x}[\ell(t)/A(t)] /\ell(x) \geq C^{-1}/A(x)$.  
     By (\ref{1/A}) (S) also follows from the inequality  $K(x)/A(x)\leq C H(x)/\ell(x)$   to be valid for all large $x$.
  \end{rem}
  
     %Rem5
  \begin{rem}\label{rem5} Let (Hab) hold. The condition  $r_-(x)/r_+(x)\to 0$ is equivalent to $A(x)/m(x)\to 1$ and in this case $\ell(x)/m(x)\to1$. Combining this fact with 
 our Theorem we  infer that all these conditions are equivalent to one another,  and  are satisfied if and only if  
  \beqn\label{U/-/+}
\lim_{x\to \infty}  \frac{U(-x,-x+h]}{U(x,x+h] } =0  \quad \mbox{  for each/some} \; h>0.
  \eeqn
 The equivalence  stated first follows from $1/A =r_+-r_-$ and $1/m= r_++r_-$.
  If $m(x)/A(x) \to 1$, then  $m'(x)= \big[m(x)/A(x)\big]^2H(x) \sim H(x)$, whence $m(x) \sim \ell(x)$ as asserted.
  \end{rem}

  %Rem6
  \begin{rem}\label{rem6}  Put $\ell_+(x)= \int_0^x (1-F(t))dt$. Then  according to \cite[Corollary 2]{Ec2}  $S_n< 0$  only finitely many times with probability one  or what amounts to the same \cite[Theorem XII.2.1]{F},
 \v2
  $(\sharp)$ \quad $\lim S_n = \infty$ \; a.s.
 \v2\noindent
    if and only if     $\int_{-\infty}^{-1} |y|dF(y) /\ell_+(-y) <\infty$. Under (Ha) this summability condition is equivalent to 
   $\int_1^\infty F(-x)dx/\ell_+(x)<\infty$; 
hence  $\liminf S_n = -\infty$  a.s.   if $\liminf F(-x)/H(x)>0$ (as alluded to previously). 
Under (Hab) it also holds  that 
  $(\sharp)$ implies (\ref{U/-/+})    (but the converse is not true). 
Indeed,   if $\int_1^\infty F(-y)dy/\ell_+(y)<\infty$, then $\int_0^xF(-y)dy\big/\ell_+(x)\to 0$, so that  $A(x)\sim \ell(x)$, which is equivalent to (\ref{U/-/+})  in view of Remark \ref{rem5}.
   \end{rem}
  \v2
\v2

Corollary \ref{cor2} follows immediately from Theorem  (in fact from Corollary \ref{cor1})  as a special case so that we have only to prove   the latter. However, we prove them separately. An obvious reason for this is the importance of Corollary \ref{cor2} and the simplicity of its proof. There is another reason.  Their proofs   are both made along the lines 
of the proof given  by  Erickson \cite{Ec}, which is however  directly applicable only for the proof of Corollary \ref{cor2}.  In fact, we give two proofs of Corollary \ref{cor2}, one 
uses a Fourier cosine representation of $U$ as in \cite{Ec}, while the other  uses the corresponding  Fourier sine representation, unlike \cite{Ec}.  The proof of Theorem 
 is carried out  by  combining these two approaches by  modifying the arguments  in  \cite{Ec}.

We  shall give  proof of Corollary \ref{cor2}  in Section 2.   Another proof of it will be provided  in Section 3;  to this end we shall  describe  the main steps of  the proof given  in \cite{Ec}   under $(*)$,  which  will also prepare for  the proof of Theorem  that will be given in Section 4.

%Section2
\section{A lemma and Proof of Corollary \ref{cor2}}
 Corollary \ref{cor2}  follows from  the arguments of \cite{Ec} if  we prove Lemma \ref{lem1}  below. 
Suppose that $X$ is non-negative and non-arithmetic  and $\ell$ is s.v. 
It  follows that as $x\to\infty$, 
\beqn\label{rv0}
x[1- F(x)] = o(\ell(x)),\; \int_0^x tdF(t)  \sim \ell(x),
\eeqn
\beqn\label{r.v.}
 \; \int_0^x t^2dF(t) \leq  2\int_0^x t [1-F(t)]dt =o\big(x\ell(x)\big)
\eeqn
and 
\beqn\label{s.v.} \int_x^\infty \frac{tdF(t)}{\ell^2(t)} \sim   \int_x^\infty \frac{1-F(t)}{\ell^2(t)}dt = \frac{1}{\ell(x)}.
\eeqn
(The proofs of these relations are standard  in view of  the fundamental properties of regularly varying functions for which  the readers are referred to \cite{BGT} or \cite{F}.)
Let 
$$\phi(\theta) = E[e^{i \theta X}].$$ 
 the characteristic function of $X$. Then, 
 with the help of  (\ref{r.v.}) one deduces from (\ref{rv0})  that
$$\Re [1-\phi(\theta)]  = \int_0^{\infty} [1- \cos \theta x] \,dF(x) = o\big(\theta\ell(1/\theta)\big),$$
$$\Im  [1-\phi(\theta)]  = -\int_0^\infty\sin \theta x\, dF(x)
 \sim -\theta \ell(1/\theta)$$
as $\theta \to 0$. 
Hence 
\beqn\label{phi}
1-\phi(\theta) = -i \theta \ell(1/|\theta|)\{1+o(1)\}  \qquad (\theta \to 0).
\eeqn
In \cite{Ec}    ($*$)  is used to have $ \Re [1-\phi(\theta)]  \sim \theta L(1/\theta)$, which cannot be obtained under the present assumption. Fortunately what is actually needed is the  result given by Lemma \ref{lem1} below.
Let us write  $C(\theta)$ and $S(\theta)$ for the real and imaginary parts of $1/[1-\phi(\theta)]$:
$$C(\theta) = \Re \frac1{1-\phi(\theta)},\quad   S(\theta)= \Im \frac1{1-\phi(\theta)} \qquad(\tht\neq0). $$
  Note that   $\phi(\tht)\neq 0$ for all $\tht\neq 0$
since  $F$ is  non-arithmetic.
\v2
\nin
%LEM1
\begin{lemma}\label{lem1}  
$$J(\theta): = \int_0^\theta C(t) dt \sim \frac{\pi/2} {\ell(1/\theta)} \qquad \mbox{as $\theta \to +0$}.$$
\end{lemma}
\begin{proof}  On noting $J(\theta)   =\int_0^\tht  |1-\phi(t)|^{-2}dt\int_0^\infty (1-\cos tx)dF(x)$,  split the range of the inner  integral on the RHS at $M/\theta$ with any number  $M>1$  and accordingly decompose $J(\theta)$ as the sum of 
two repeated integrals. 
After interchanging the order of integration for the integral over $x>M/\theta$, this results in   $J(\theta)= J_1(\theta)+ J_2(\theta)$, where
$$J_1(\theta) = \int_{M/\theta}^\infty dF(x) \int_0^\theta \frac{1-\cos xt}{|1-\phi(t)|^2}dt,$$
$$J_2(\theta) = \int_0^\theta \frac{dt}{|1-\phi(t)|^2} \int_0^{M/\theta} (1-\cos tx) dF(x).$$ 

First we see  that $J_2(\theta)$ is negligible.   The integrand of the integral  defining  $J_2(\theta)$  is less than 
 $$ \frac{t^2/2}{\,|1-\phi(t)|^2}\int_0^{M/\theta} x^2 dF(x) =\frac1{\ell^2(1/t)} \times o\big( \theta^{-1}\ell(1/\theta)\big),$$
 where (\ref{phi}) as well as (\ref{r.v.}) is used for the equality.
 This immediately yields   $J_2(\theta)\ell(1/\theta)\to 0$.
 
In order to estimate $J_1(\theta)$  we  break the inner integral into three parts  
and   using  (\ref{phi}) we deduce that   uniformly for $x>M/\theta$, as $\theta \downarrow 0$,
$$\int_0^{1/Mx} \frac{1-\cos xt}{|1-\phi(t)|^2}dt \leq  x^2\int_0^{1/Mx} \frac{dt}{\ell^2(1/t)}  \sim  \frac{x/M}{\ell^2(x)},$$  
$$\int_{M/ x}^{\theta} \frac{1-\cos xt}{|1-\phi(t)|^2}dt \leq  \int_{M/x}^1 \frac{\{2+o(1)\}dt}{t^2\ell^2(1/t)}
 \sim  \frac{2x/M}{\ell^2(x)}$$
 and 
 $$\int_{1/Mx}^{M/x}  \frac{1-\cos xt}{|1-\phi(t)|^2}dt \sim \frac1{\ell^2(x)} \int_{1/Mx}^{M/x}  \frac{1-\cos xt}{t^2}dt = \frac{x}{\ell^2(x)} \int_{1/M}^{M}  \frac{1-\cos t}{t^2}dt. $$
 Hence
 $$J_1(\tht) = \int_{M/\tht}^\infty \frac{xdF(x)}{\ell^2(x)} \Big\{\frac\pi{2} +O(1/M)\Big\}.$$
and  in view of (\ref{s.v.})  
 $J_1(\theta)\ell(1/\theta)  \to \frac12 \pi$, for $M$ can  be made  arbitrarily large.
This concludes the proof of the lemma. 
 \end{proof}

In  \cite{Ec} the auxiliary condition ($*$) is used to obtain the estimate of Lemma \ref{lem1} (as well as of (\ref{phi})) and once it is established  the proof of (\ref{m_res}) proceeds without directly using $(*)$ (see   Section 3.1  of this paper). With Lemma \ref{lem1} at hand we can accordingly  follow the arguments given in Section 5  of \cite{Ec}  to complete the proof of Corollary \ref{cor2}. 

%SECTION3
\section{Proof of Corollary \ref{cor2} based on  the estimate of $S(\tht)$}

 In  this section, comprising  three subsections,  we point out that  the proof of Corollary \ref{cor2} can be based on   (\ref{phi}) rather than   Lemma \ref{lem1}; this way  is rather simpler, being   performed without
 resorting any estimate of $C(\theta)$ except for its summability in a neighbourhood of zero that is valid quite generally  (see (\ref{eq10})). The exposition given in this section also prepares for the proof of Theorem.  In the first  two subsections the entire mass of $F$ is supposed to concentrate on $[0,\infty)$. 
   In the last subsection $F$ is allowed to have  a mass on the negative half-line and we obtain a general analytic  result  that  is fundamental in our proof of Theorem. 
 \v2
 
{\bf 3.1.}  In this subsection  we describe the main steps of the proof of (\ref{m_res})  according to  \cite{Ec}.
For a real   function $g\in L_1(\R)$  let $\hat g$ be its Fourier  transform: 
$$\hat g(x)= \int_{\mathbb{R}} e^{-ix\tht}g(\tht)d\tht.
$$
Then an elementary manipulation leads to   
\beqn\label{eq1}
\int_{\R} e^{i\theta \xi} \hat g(\xi)U\{x+d\xi\} = \int_\R U\{dy\}\int_\Re e^{i(y-x)u}g(\tht-u)du,
\eeqn 
provided that the integral on the LHS is absolutely convergent for which it is sufficient that
\beqn\label{eq0}
 \hat g(x)= O(x^{-2}) \quad (|x|\to\infty)
 \eeqn
since $U(x,x+1]$ is bounded. 
 Below  $g$ is always assumed to be continuous and piecewise smooth ($C^2$-class  is enough) up to the end-points and vanish outside a compact set, which in particular ensures (\ref{eq0}).

 Take the particular function  $g=g_a$, $a>0$ defined by 
$$g_a(\tht)=\left\{ \begin{array} {ll} a^{-1}(1-|\tht|/a)\quad & |\tht|\leq a, \\
0 &|\tht|>a, 
\end{array}\right.
$$
 whose  Fourier transform is equal to  
$\hat g_a(x) =2(1-\cos(ax))/ a^2x^2$.
It is shown  \cite[Lemma 8]{Ec} that given a constant $\nu\geq 0$ and  a family of locally finite measures  $(\mu_x)_{x>0}$ on $\R$, 
 \beqn\label{ET} 
\begin{array}{ll} 
\mbox{ \it if}\quad  \forall a>0,  \forall \tht\in \R,\quad {\displaystyle
\lim_{x\to\infty} \frac1{2\pi}\int_\R e^{i\tht \xi}\hat g_a(\xi)\mu_x\{d\xi\}  =}\nu g_a(\tht),\\[3mm]
\mbox{\it then}\quad \forall h> 0, \forall y\in \R, \quad \lim_{x\to\infty}\mu_x\{[y, y+h)\} = \nu h.
\end{array} \qquad
\eeqn 
(Here the trivial case $\nu=0$ is included,  although it is not explicitly stated in \cite{Ec}.)
On applying this result to $\mu_x\{d\xi \} =\ell(x) U\{x+d\xi\}$ it suffices  for the proof of Corollary \ref{cor2} to show that for each $a>0$ and $\tht \in \mathbb{R}$, as $x\to \infty$
\beqn\label{eq4} 
\frac1{2\pi}\int_\R e^{i\tht \xi}\hat g_a(\xi)U\{x+d\xi\} \sim \frac{g_a(\tht)}{\ell(x)}.
\eeqn
Here and in the sequel  the sign $\sim$ is understood in the obvious way for the case $g_a(\tht)=0$.

In order to verify (\ref{eq4})  the relation (\ref{eq1}) or its variant will be  used. On  formal level  one  might interchange the order of
the repeated integral on its RHS and derive  $\int_\R e^{iyu}U\{dy\} = \sum_{0}^\infty \phi^n(t)= 1/[1-\phi(u)]$ to find
\beqn\label{eq2}
\int_{\R} e^{i\theta \xi} \hat g(\xi)U\{x+d\xi\} =  \int_\R  \frac{g(\tht-u)}{1-\phi(u)}e^{-ixu} du.
\eeqn 
We shall see this identity is valid but the difficulty of direct  approach  would be obvious: 
actually  $U\{\R\}=\infty$, so that  the Fourier-Stieltjes transform of $U(d\xi)$ does not exist in a usual sense.   

Following \cite{FO}   
we  make the decomposition  $U\{I\} =V\{I\}+ V^*\{I\}$, where
$$V\{I\} = \frac12 \big[U\{I\}+ U\{-I\}\big], \quad V^*=\frac12 \big[U\{I\}- U\{-I\}\big].$$
For the proof of (\ref{eq4}) we may replace  $U$  by $2V$ therein, since  $ U\{-x- d\xi\}= 0$ for $\xi >-x$---entailing   
\beqn\label{U/V}
U\{x+d\xi\} =2V\{x+d\xi\} \qquad (\xi>-x)
\eeqn
  and  by (\ref{eq0}) the integral over $\xi\leq -x$,   in either case, is $O(1/x)$, negligible in comparison to $1/\ell(x)$.  From Eq(15) of \cite{FO} one derives the identity
    \beqn\label{eq20}
\int_{\R} e^{i\theta \xi} \hat g(\xi)V\{x+d\xi\} =  \int_\R  g(\tht-u)C(u)e^{-ixu} du
\eeqn 
of which we shall give a self-contained proof  (see Lemma \ref{lem3} below).

If one applies  Lemma \ref{lem1} of the preceding section, it is easy to see that
\beqn\label{E6}
\int_{|u|<M/x}   g_a(\tht-u)C(u)e^{-ixu} du \sim \pi \frac{g_a(\tht)}{\ell(x)}
\eeqn
for each $M$ fixed. Hence,  if one can show that as $x\to\infty$
\beqn\label{E7}
\bigg|\int_{|u|\geq M/x}   g_a(\tht-u)C(u)e^{-ixu}du\bigg| < \frac{C}{M\ell(x)}, 
\eeqn
\noindent
then (\ref{eq4})  follows because of (\ref{U/V})  and (\ref{eq20}).  

Below we provide  a   proof of (\ref{E7}) (somewhat simpler than that in \cite{Ec}), which we shall follow  in the proof of Theorem with obvious adaptation and additional arguments. 
 
Put  $\om_x(\tht) =1- E\big[ e^{i\tht X};0\leq X < x]$
and  $r_x(\tht) = E[e^{iX\tht};  X\geq x]$
so that  
$$1-\phi(\tht)= \om_x(\tht) - r_x(\tht).$$
Since $\Re\phi(\tht)< 1$ for $\tht\neq0$ and   $|r_x(\tht)|\leq 1-F(x)=o\big(\ell(x)/x\big)$    it follows from (\ref{phi}) that for each $b>0$ there exist positive constants  $\de$ and $c$ such that as $x\to\infty$
\beqn\label{om/r0}
\left\{\begin{array}{ll}
|\om_x(\tht)| \geq c \quad\quad &\mbox{ for} \;\; \de\leq  |\tht| <b,\\
\om_x(\tht) = -i\tht \ell(1/|\tht|)\{1+o(1)\} \quad\quad &\mbox{uniformly for} \;\; 1/x < |\tht| < \de,
\end{array} \right.
\eeqn
where $o(1)\to 0$ as $\tht\to 0$ under $|\tht|>1/x$ and $|o(1)| <1/2$. Note  that     $r_x(\tht) =o(\om_x(\tht))$ uniformly for $1/x < |\tht| < b$ as $x\to\infty$ and 
\beqn\label{Decom0}
\frac1{1-\phi(\tht)} =   \frac{1}{\om_x(\tht)} +\frac{r_x(\tht)}{\om^2_x(\tht)}\bigg[1+\frac{r_x(\tht)}{1-\phi(\tht)}\bigg].
\eeqn
Write $H(x)= 1-F(x).$
By (\ref{om/r0})  it then follows that 
\beqn\label{est_err0}
\int_{1/x}^b \bigg|\frac{r_x(\tht)}{\om^2_x(\tht)}\bigg|d\tht \leq \int_{1/x}^b \frac{CH(x)  d\tht}{\tht^2\ell^2(1/\tht)} =\int_{1/b}^x \frac{CH(x) dt}{\ell^2(t)} \sim \frac{CxH(x)}{\ell^2(x)} = o\bigg(\frac{1}{\ell(x)}\bigg),
\eeqn
showing that the real part of the second term on the RHS of (\ref{Decom0}) makes only a  negligible contribution to the integral in (\ref{E7}). The  derivative  $\om'_x$  admits the bound $$|\om'_x(\tht)| = \big|E\big[Xe^{i\tht X}; X <x\big]\big| \leq \int_0^x ydF(y) \leq \int_0^x(1-F(y))dy=\ell(x).$$
Hence
  \beqn\label{E5}
 \int_{M/x}^b\bigg| \frac{\om_x'(\tht)}{\om^2_x(\tht)}\bigg| d\tht \leq 
 \int_{M/x}^b \frac{\ell(x)d\tht}{[\tht\ell(1/\tht)]^2} + \frac{\ell(x)b}{c^2}=  \ell(x) \int^{x/M}_{1/b}\frac{dy}{\ell^2(y)} +O(\ell(x)) \leq \frac{Cx}{M\ell(x)}.
\eeqn
Now on noting that $\big[\Re(1/\om_x(\tht))\big]' =- \Re \,\big(\om'_x(\tht)/[\om_x(\tht)]^2\big)$
 integrating by parts  readily leads to  (\ref{E7}) as required.

%%\2
\v2
{\bf 3.2.} Owing to (\ref{ET}) 
it suffices  for the proof of Corollary \ref{cor2} to show that as $x\to\infty$
\beqn\label{eq4*} 
\frac1{\pi}\int_\R e^{i\tht \xi}\hat g(\xi)V^*\{x+ d\xi\} \sim \frac{g(\tht)}{\ell(x)}
\eeqn
since  $U(x+d\xi) = 2V^*(x+d\xi)$ ($x>1, \xi< -x$).
 In the next subsection we show  the identity
\beqn\label{C} \int_{\R} e^{i\theta \xi} \hat g(\xi)V^*\{x+d\xi\} =  i\int_\R  g(\tht-u) S(u)e^{-ixu} du
\eeqn
valid for every non-arithmetic $F$ on $\mathbb{R}$---this is the exact counterpart of the formula (\ref{eq20}).  
By (\ref{phi}) we have
\beqn\label{A}
S(u) \sim 1/u\ell(1/|u|) \; \;\; (u\to 0).
\eeqn
Since $S$ is an odd function, 
the RHS of (\ref{C}) restricted to $|u|<M/x$, $M>1$ is written as
\beqn\label{g/SS}
 g(\tht) \int_{|u|<M/x} S(u)\sin xu\,du + i\int_{|u|<M/x}\big[ g(\tht- u)-g(\tht)\big]S(u) e^{-ix u}du,
 \eeqn
and  easily   evaluated to be  $[\pi/\ell(x)]\{g(\tht)+o(1) + O(1/M)\}$ in view of $\int_{|u|<M/x} \sin xu \, du/u = \pi +O(1/M)$. 
By (\ref{om/r0}) and  (\ref{E5}) we see that the other integral is bounded in absolute value  by a constant multiple of $M^{-1} /\ell(x)$ in the same way as above. Thus (\ref{eq4*}) follows, for $M$ can be made arbitrarily large, and 
the proof of Corollary \ref{cor2} will be  complete if we can   show (\ref{C}).
\v2
{\bf 3.3.} {\it Proofs of (\ref{eq20}) and (\ref{C}).}  In this subsection $X$ may take negative values.  For application in the next section we formulate the results as  two lemmas. 

%Lem2
\begin{lemma}\label{lem2}  If  $F$ is  non-arithmetic and transient,   then
 (\ref{C}) holds and  
  $\int_0^1|u S(u)|du <\infty$. 
\end{lemma}
\pf
Put   
$$V^*_s\{dx\} =\frac12 \sum_{n=0}^\infty s^n \big [F^{n*}\{dx\} - F^{n*}\{-dx\} \big]\quad\;\;  (0<s<1).$$  
Then  in the identity (\ref{eq1})   $U$ may be replaced by $V^*_s$ simultaneously on both sides of it.   The order of the repeated integral on the RHS of the resulting identity  may be interchanged since  the total variation of $V^*_s$ is finite. 
 Observing 
   \beqn\label{eqIm}
  \int_\R e^{i u y} V^*_s\{dy\}  =\frac{1}{2}\bigg[\frac1{1-s\phi(u)} - \frac1{1-s\phi(-u)}\bigg]
 =i\Im \frac{1}{1-s \phi(u)},
 \eeqn 
we thereby  obtain 
\beqn\label{eq6}
\int_{\R} e^{i\theta \xi} \hat g(\xi)V^*_s\{x+d\xi\} 
=i\int_\R e^{-ixu}g(\tht-u)\,\Im\frac1{1-s \phi(u)}du.
\eeqn
Thus  formally taking limit as $s\uparrow 1$ gives (\ref{C})  which may be written as 
\beqn\label{eq2*}
\int_{\R} e^{i\theta \xi} \hat g(\xi)V^*\{x+d\xi\} =  i\int_\R  e^{-ixu} g(\tht -u)\, \Im\frac{1}{1-\phi(u)} du.
\eeqn 
 The convergence of the left-hand integral in (\ref{eq6}) to that of (\ref{eq2*})  is obvious from  (\ref{eq0}), i.e., $\hat g(x)=O(x^{-2})$.
We must verify  
\beqn\label{eq7}
 \int_\R e^{-ixu}g(\tht -u)\, \Im\frac1{1-s \phi(u)}du \, \longrightarrow \int_\R  g(\tht-u) S(u)e^{-ixu} du \quad \mbox{as \, $s\uparrow1$}.
\eeqn
Observe that  the odd part of 
$$e^{-ixu}g(\tht - u) = \big[g(\tht-u) -g(\tht)\big] e^{-ixu}  + g(\tht)e^{-ixu}$$
 is dominated in absolute value by a constant multiple of $|u|$  (because of the assumption on   $g$) while,
 $\Im [1- s\phi(u)]^{-1}$ being odd, the integral  of the even part multiplied by  $\Im [1-s\phi(u)]^{-1}$ (as well as   $S(u)$) vanishes.  We then infer that 
   (\ref{eq7})  follows  if we can  show that
   \beqn\label{eq8}
   \lim_{s\uparrow 1} \int_{0}^1 \bigg|  \Im\frac1{1-s \phi(u)}  -S(u)\bigg|udu =0.
   \eeqn
 
  As is well known the transience of $F$ implies
\beqn\label{eq10}
\int_0^1 C(u)du<\infty
\eeqn
 (see a remark immediately before  Lemma \ref{lem6}). 
    We claim that  for some $M\geq 1$
\beqn\label{eq9}
\bigg| u\, \Im\frac1{1-s \phi(u)}\bigg| \leq \frac{uE|\sin uX|}{|1-\phi(u)|^2} 
\leq M C(u) \quad (0<u<1, 0<s<1). 
\eeqn
The first inequality is obvious.  For the proof of the second,    
observe   $1-\Re\phi(u) = 2E\sin^2 ({\textstyle \frac12} uX)$ and   $(\sin 1) uE[|X|; |X| \leq 1] \leq E |\sin uX| \leq 2\big[E \sin^2 ({\textstyle \frac12} uX)\big]^{1/2}$. Using these one then  deduces
$$uE|\sin uX| \leq 2\big[ u\big/E|\sin uX| \big] (1-\Re\phi(u)) < \de^{-1} \Re\big[1-\phi(u)\big] \quad (0< u<1)$$
for some $\de>0$,   
verifying  (\ref{eq9}). Now, 
  by  the dominated convergence  (\ref{eq10}) implies  the summability $\int_0^1|uS(u)|du<\infty$, and hence   (\ref{eq8}).  
  The proof of Lemma \ref{lem2} is complete.  \qed 
 \v2
 The next lemma that is used in Section 4.2 (apart from its use in the preceding subsection)  is stated here as an appropriate place, although its proof partly depends on the results of Section 4.1 that follows.
 %Lem3
 \begin{lemma}\label{lem3}  Suppose that $F$ is non-arithmetic and satisfies (Ha).  Then  (\ref{eq10}) implies 
 (\ref{eq20}). 
  \end{lemma}
\pf  The result can be directly derived from a somewhat more general result  \cite[Eq(15)]{FO}. The proof presented below, an adaptation of that in \cite{FO}, is simpler because of the extra assumption of (Ha).  

Analogously to the preceding proof, putting 
$   V_s\{dx\} = \frac12\sum_{n=0}^\infty  s^n[F^{n*}\{dx\} +  F^{n*}\{-dx\}]$ ($0<s<1$), 
 one sees that $\int_\R e^{i u y} V_s\{dy\}  
 =\Re \big({1-s \phi(u)}\big)$
and
  \beqn\label{eq60}
\int_{\R} e^{i\theta \xi} \hat g(\xi)V_s\{x+d\xi\} =\int_\R e^{-ixu}g(\tht-u)\,\Re\frac1{1-s \phi(u)}\,du
\eeqn
 similarly  to (\ref{eqIm}) and (\ref{eq6}), respectively. 
We let   $s\uparrow 1$ in (\ref{eq60}) to  see (\ref{eq20}) which we may write as
\v2
{\rm Eq}(\ref{eq20}):  \qquad  
${\displaystyle \int_{\R} e^{i\theta \xi} \hat g(\xi)V\{x+d\xi\} =  \int_\R  e^{-ixu}  g(\tht-u) \Re\frac1{1-\phi(u)}du.}$
\v2\noindent
 The left-hand integral in (\ref{eq60})   is easily disposed of as before. For the right-hand one,   verification may be required.  By the inequality
 $|1-s\phi|^2 \geq (1-s)^2 +s^2|1-\phi|^2$,  we have for $\de>0$ 
\beqn\label{Olnstn}
\int_{-\de}^\de \Re\frac1{1-s \phi(u)}du \leq \int_{0}^\de \frac{2(1-s)}{(1-s)^2+ s^2|1- \phi(u)|^2} du+  \int_{0}^\de \frac{2s\big(1-\Re\phi(u)\big)du}{(1-s)^2+ s^2|1- \phi(u)|^2}.
\eeqn
 By (\ref{eq10}) the second integral on the RHS converges  to zero as $\de\to0$ uniformly for $s$. The same convergence also  is true  of the first integral  under   $\lim |1-\phi(u)|/|u|=\infty$ which (Ha) together with (\ref{eq10}) entails
  (cf. Lemmas \ref{lem4} and \ref{lem5} in Section 4.1).  
 Since $g$ vanishes outside a compact set,  these together show Eq(\ref{eq20}). \qed
 \v2
By what is observed right after (\ref{Olnstn}) the integral on its LHS is bounded as $s\uparrow1$, which in particular shows that $F$ is transient (cf. e.g., \cite[Section XVIII.7]{F}).

%SECTION4 
\section{Proof of Theorem.}

In this section  $X$  assumes both positive and  negative values and  $F$ is supposed to be non-arithmetic.   
As before,  let   $C$  and $S$ be the real and imaginary parts of $1/(1-\phi)$  so that 
\beqn\label{C+S}
1/[1-\phi(\tht)] = C(\tht)+ iS(\tht)\quad\; \mbox{for \; $\tht\in \R\setminus\{0\}$}.
\eeqn
 \v2
  
 {\bf 4.1.} {\it Preliminary lemmas.} In this subsection we shall suppose   (Ha) to  hold. 
 %Lem4
\begin{lemma}\label{lem4}  If (Ha) holds, then as $\tht\to 0$
\beqn\label{29}
1-\phi(\tht) =-i\tht A(1/|\tht|)\{1+o(1)\}.
\eeqn
\end{lemma}
 This lemma is contained in  the criteria for relative stability of $F$ obtained by Maller \cite[Theorem 1]{M79}.
    Here we present  a  proof, it being  simple enough---Maller \cite{M79} derives (\ref{29})  directly from the relative stability of $F$, but its derivation from (Ha) is simpler---and the same arguments being employed  implicitly in the sequel.  
As mentioned in Section 1    $A$ is s.v.\;under (Ha).
By skew symmetry we may consider only  the case $\tht>0$. Since $\lim_{y\to\infty} \int_0^{y}K(x)\cos \theta x\,dx$ exists, we have for $\tht>0$
$$ \Im [1-\phi(\tht)] = -\int_R \sin \tht x\, dF(x)= -\tht \int_0^{\to \infty} K(x) \cos \tht x dx,$$  
and observe  that  on the one hand by using the monotonicity of $F(-x)$ and $1-F(x)$,
\[
\bigg|\int_{1/\tht}^{\to \infty}  K(x) \cos \tht x \,dx\bigg| 
\leq \frac{\pi H(1/\tht)}{\tht}  =o(A(1/\tht)),
\]
and on the other hand 
 $|\int_0^{1/\tht} K(x)(1- \cos \tht x) dx| \leq \tht \int_0^{1/\tht} xH(x)dx =o(A(1/\tht))$, to obtain 
\begin{eqnarray}\label{<}
\int_0^{\to \infty}  K(x) \cos \tht x \,dx    &=& A(1/\theta) -\int_0^{1/\tht}K(x)(1-\cos \tht x)dx + \int_{1/\tht}^{\to \infty}  K(x)\cos \tht x \,dx \nonumber\\
&=&A(1/\tht) \{1+o(1)\}.
\end{eqnarray}
This concludes $\Im [1-\phi(\tht)] \sim -\tht A(1/\tht)$ ($\tht\to0$).  Similarly, 
we deduce   $\big|\Re [1-\phi(\tht)] \big|=\big|\tht \int_0^\infty H(x)\sin \tht x\, dx\big| = o(\tht A(1/\tht))$. Thus  (\ref{29}) has been verified.  

\v2
We are going to prove several  lemmas concerning $1/[1-\phi(\theta)]$.  Instead of (\ref{r.v.}) we have
\beqn
\label{C/A}
\int_{0} ^x t^2d(-H(t)) \leq 2\int_0^x tH(t)dt =  o\big(xA(x)\big).
\eeqn 
Note that $1/A(x) \leq \int_x^\infty H(t)dt/ A^2(t)\leq \infty$ because of  (\ref{1/A}) and that   $t/A^2(t)$ is increasing in a neighborhood of $\infty$. Then one 
sees that  
\beqn\label{dF/A}
\int_{x}^\infty \frac{td(-H(t))}{A^2(t)}  
=   \int_x^\infty\frac{H(t)}{A^2(t)}dt\{1+o(1)\}, 
\eeqn  
where the two integrals   are simultaneously  finite or infinite.     Recall  that  (Hb)  is the integrability condition 
$ \int_{x_0}^\infty H(t) dt/A^2(t)<\infty$. 

%Lem5
\begin{lemma}\label{lem5} Suppose  (Ha)  holds.  Then  (Hb) is equivalent to $\int_0^1 C(\tht)d\tht<\infty$ and implies
\beqn\label{C/H/A}
\int_0^{\tht} C(u)du \sim \frac{\pi}{2m(1/\tht)} = \frac{\pi}{2}\int_{1/\tht}^\infty \frac{H(x)}{A^2(x)}dx  \qquad (\tht \downarrow 0).
\eeqn 
\end{lemma}
\pf  We compute $J(\tht):=\int_0^\tht C(u)du$  by following the proof of Lemma \ref{lem1}. Given $M>1$ and  $0<\tht <1/x_0$,
 define  $J_1$ and $J_2$ as in its proof, but with  $dF(x)$ replaced by $d(-H(x))$ so that $J(\tht)=J_1(\tht)+J_2(\tht)$. Then on using (\ref{29}) and (\ref{C/A})
\beqn\label{J2}
J_2(\tht) \leq  \int_{0}^{\tht} \frac{dt}{t^2A^2(1/t) \{1+o(1)\}}\int_0^{M/\tht} \frac{(tx)^2}2 d(-H(x)) =o\bigg(\frac1{A(1/\tht)} \bigg),
\eeqn
while by (\ref{29}) again   one sees in the same way as before  that uniformly for  $x\geq M/\tht$ 
\beq
 \int_{0}^\tht  \frac{1-\cos xt}{|1-\phi(t)|^2}dt  &=&  x\int^{M}_{1/M} \frac{(1-\cos u)du}{u^2A^2(x/u)\{1+o(1)\}} + \bigg(\int^{1/Mx}_0+ \int_{M/x}^\tht \bigg)\frac{(1-\cos xt)dt}{t^2A^2(1/t)\{1+o(1)\}} \\
 &=&
 \frac{x}{A^2(x)}\bigg\{\frac\pi 2 +o(1) + O\Big(\frac1{M}\Big) \bigg\}
 \eeq
and  hence
 \beqn\label{J1}
   J_1(\tht) =  \int_{M/\tht}^\infty  \frac{xd(-H(x)))}{A^2(x)}\bigg\{\frac\pi 2+ o(1) + O\Big(\frac1{M}\Big)\bigg\}.
  \eeqn
Finally we combine (\ref{J2}), (\ref{J1})  
and  (\ref{dF/A})  to conclude that  $J(\tht)<\infty$ if and only if (Hb) holds and   either one implies  (\ref{C/H/A}).
\qed

\v2
The summability, $\int_0^1 C(t)dt<\infty$, is necessary and sufficient for  the transience of $F$. [The necessity part follows from the well-known criterion for  transience as given in \cite{F}, whereas we have already observed that the sufficiency part follows from the same criterion under (Ha) just after  (\ref{Olnstn}). The sufficiency part, true in general, is obtained by Ornstein \cite{O}  (cf. also \cite{St}), of which, however, the proof is quite involved and not found in standard textbooks of probability theory.]     By Lemma \ref{lem2} the above summability of $C(u)$ implies  $\int_0^1 |\tht S(\tht)|d\tht<\infty$. For clarity and convenience of later citation,  we state   these consequences of Lemma \ref{lem5} as a lemma. 
%Lem6
\begin{lemma}\label{lem6} Suppose  $(Ha)$ to hold.  Then  (Hb) is a necessary and sufficient condition for  the transience of $F$, and implies 
\beqn\label{tS}
\int_0^1 \big[ C(\tht) + |\tht S(\tht)|\big] d\tht <\infty. 
\eeqn
\end{lemma}
\v2

We need to get the  relation corresponding to (\ref{E6}) for the present situation 
where $\ell$ must be replaced by $m$. To this end the  inequality (\ref{E5})---with $\om_x(\theta)$ given below---is inadequate 
and we give an appropriate inequality in  Lemma  \ref{lem8} shortly. In preparation for the proof of Lemma  \ref{lem8}  we state some
facts  analogous  to  the last part of Section 3.1.
 
 We make the decomposition
 $1-\phi(\tht)= \om_x(\tht) - r_x(\tht)$, where
\[
\om_x(\tht) = 1- E\big[ e^{i\tht X}: |X|<x  \big], \quad r_x(\tht)= E\big[ e^{i\tht X}: |X|\geq x  \big].  
\]
Since $\Re\phi(\tht) < 1$ for $\tht\neq0$,  $|r_x(\tht)|\leq H(x)=o(A(x)/x)$ ($x\to\infty$), and 
$1-\phi(\tht)\sim -i\tht A(1/|\tht|)$ ($\tht\to 0$), it follows that   for each $b>0$,  $ r_x(\tht) = o(\om_x(\tht))$ as $x\to\infty$ uniformly  for $1/x<|\tht|<b$
and  there exists  $c>0$ and $\de>0$ such that
\beqn\label{om/r}
\begin{array}{ll}
|\om_x(\tht)|\geq c \qquad  &\mbox{for} \;\; \de \leq  |\tht|< b,\\
\om_x(\tht) = -i\tht A(1/|\tht|)\{1+o(1)\}, \quad \quad &\mbox{uniformly for} \;\; 1/x < |\tht| < \de,
\end{array}
\eeqn
where 
 $o(1)\to 0$ as $\tht\to 0$ and $|o(1)|<1/2$ (as in  (\ref{om/r0})). 
By (\ref{om/r}) 
we infer as in (\ref{est_err0}) that 
\beqn\label{est_err}
\int_{1/x}^b \bigg|\frac{r_x(\tht)}{\om^2_x(\tht)}\bigg|d\tht
=  o\bigg(\frac{1}{A(x)}\bigg),
\eeqn
so that if $f(\tht)$ is piecewise smooth, then  as $x\to\infty$
\beqn\label{Decom}
\int_{M/|x| <|\theta <b} \frac{f(\tht)e^{ix\tht}}{1-\phi(\tht)}   d\tht 
= \int_{M/|x| <|\theta <b} \frac{f(\tht)e^{ix\tht}}{\om_x(\tht)}   d\tht +o\bigg(\frac1{A(x)}\bigg).
\eeqn

We have the bound $|\om'_x(\tht)|\leq \ell(x)$ as before,  which however yields only the bound  (\ref{E5}) with  $C\ell(x)/MA^2(x)$ in place of $C/M\ell(x)$ in its right-most member, a bound insufficient for the proof of Theorem. This issue will be  cleared up  in Lemma \ref{lem8} by using the following

%Lem7
\begin{lemma}\label{lem7}   If  (Ha) holds, then 
\[
|  \om_x'(\tht)| =  A(x)\big\{1+ o\big(\sqrt{|\tht|x} \,\big)\big\} \quad \mbox{as\,
$x\, \to \, \infty$\, uniformly for\, $ |\tht| >1/x$}.
\] 
\end{lemma}
\begin{proof} 
Performing  differentiation  we have
\[
\om_x'(\tht) 
= -  i\int_{-x}^x ye^{i \tht y} \, dF(y) = - i\int_{-x}^x ydF(y) + i \int_{-x}^x y(1 - e^{i \tht y}) dF(y).
\]
The first integral on the RHS is asymptotically equivalent to  $A(x)$ under (Ha). 
As for the second one, on noting  $|e^{i y\tht}-1| \leq 2\sqrt{|\tht|y}$ ($y>0$),   observe
\beq
\bigg|\int_{-x}^x y (1-e^{i \tht y}) dF(y)\bigg| \leq 2 \sqrt {|\tht|} \int_{-x}^x |y|^{3/2} \, dF(y)
 &\leq& 3 \sqrt {|\tht|} \int_{0}^x \sqrt y \, H(y) dy\\
 & =& \sqrt{|\tht|x}\times o(A(x)),
\eeq
where  the equality is obtained by  $\sqrt yH(y)=o\big(A(y)/\sqrt y\big)$.
The proof is finished.
\end{proof}

\v2
%Lem8
\begin{lemma}\label{lem8} \, Suppose  that  (Ha) is satisfied. Then for any positive constant $b$ and any function $f(\theta)$  that is piecewise  continuously differentiable (including boundaries),   there exists a constant $C$ such that for any $M>2\pi$  and $|x|>x_0$,
\beqn\label{Lem7}
\bigg|\int_{M/|x| <|\theta <b} \frac{f(\tht)e^{ix\tht}}{1-\phi(\tht)}   d\tht\bigg| \leq \frac{C}{M A(|x|)}.
\eeqn
\end{lemma}
\begin{proof}  We have only to consider  the case $x>x_0$, the other one being treated in the same way. Let  $x>0$ and $f(\tht)=0$ for $\tht >b >0$. 
To the integral on the RHS of (\ref{Decom})  we apply integration by parts to have
\beqn\label{1/om}
\int_{M/x}^b \frac{f(\tht)e^{ix\tht}}{\om_x(\tht)}   d\tht = O\bigg(\frac1{x\om_x(M/x)}\bigg) + \frac{1}{ix}
\int_{M/x}^b \frac{f(\tht)\om_x'(\tht)e^{ix\tht}}{\om^2_x(\tht)} d\tht. 
\eeqn
The contribution  to the second term of the last integral restricted on $\de<\theta\leq b$ is negligible, while by Lemma \ref{lem7} the contribution  of the other part of the integral is  dominated in absolute value by a constant multiple of
\[
\frac{A(x)}{x} \int_{M/x}^\de \frac { 1+ o\big(\sqrt{\tht x} \,\big)\,}{\tht^2 A^2(1/\tht)}  d\tht
= 
\frac{A(x)}{x} \int_{1/\de}^{x/M} \frac{1+ o\big(\sqrt{x/t}\,\big)}{A^2(t)} dt\sim  \frac{1}{M A(x)}.
\]
This together with  $x|\om_x(M/x)|  \sim MA(x)$ concludes the proof. 
\end{proof}
\v2
Applying  the bound of Lemma \ref{lem8} for an even function $f$ and considering it  for $- x$ in place of $x$ as well we infer that
\beqn\label{C-S}
\bigg|\int_{M/x}^1 f(\tht)C(\tht)\cos x\tht \,  d\tht\bigg| \leq \frac{C}{MA(x)} \quad\mbox{and}\quad 
\bigg|\int_{M/x}^1 f(\tht)S(\tht)\sin x\tht  \, d\tht\bigg| \leq \frac{C}{M A(x)}.
\eeqn
\v2
\v2\v2

{\bf 4.2.} {\it Proof of Theorem.}
The proof will proceed along  the same lines as the proof described in Section 3.2  of Corollary \ref{cor2}   with the help of Lemmas \ref{lem6} and \ref{lem8}.  The proof will also  rest on the general statement
 (\ref{ET})  valid independently  of the  present situation. [Except for   (\ref{ET}), the proof given below  is self-contained.] 
First note that Lemma \ref{lem6} ensures both Lemmas \ref{lem2} and \ref{lem3} being  applicable which together give
 (\ref{eq2}), or what is the same thing
\beqn\label{E}
\int_\R e^{i\tht\xi}\hat g(\xi)U\{x+ d\xi\}  = \int_\R g(\tht-u)e^{-ixu}\big[C(u)+ iS(u)\big]du.
\eeqn

Now we can easily  complete the  proof  of  Theorem. By Lemma \ref{lem5}   it follows that as $x\to\infty$
 $$\int_{-M/x}^{M/x} C(u)\cos xu\,du
\sim \pi \int_{x}^\infty \frac{H(t)}{A^2(t)}dt; 
$$
also by  (\ref{29})  $S(\tht) \sim [\tht A(1/\tht)]^{-1}$, and   using  this we  see as in the proof of Lemma  \ref{lem5}  that
$$\int_{-M/x}^{M/x}  S(u)\sin xu\,du = \frac{1}{ A(x)}\{\pi+ o(1)+ O(1/M)\}.
 $$
 Making the same argument as given at (\ref{g/SS}) and  immediately after it and employing (\ref{C-S})   in place of (\ref{E5})  (to dispose of the integral over $|u|>M/x$)  as well as the  estimates obtained right above we   compute the RHS of (\ref{E}) to see that if $g(\tht)\neq 0$,
 $$\int_\mathbb{R} g(\tht-u)e^{-ixu} C(u)du \sim g(\tht)\int_{-1}^1 \cos xu\, C(u)du \sim \pi g(\tht) \int_{|x|}^\infty \frac{H(t)}{A^2(t)}dt \quad (x \to \pm \infty)$$
  and
 $$ i\int_\mathbb{R} g(\tht-u)e^{-ixu} S(u)du \sim g(\tht) \int_{-1}^1 \sin xu\, S(u)du \sim \pm \frac{\pi g(\tht)}{A(|x|)} \quad (x\to \pm \infty).$$
 Hence,   on recalling $1/A(|x|) =\int_{|x|}^\infty K(t)dt/A^2(t)$, as $x\to \pm \infty$
 $$\int_\R e^{-i\tht\xi}\hat g(\xi)U\{x+ d\xi\}  = \pi g(\tht) \bigg(  \int_{|x|}^\infty \frac{H(t)}{A^2(t)}dt\{1+o(1)\} \pm \int_{|x|}^\infty \frac{K(t)}{A^2(t)}dt\{1+o(1)\}\bigg),$$
 which  shows the formula of Theorem in view of  (\ref{ET}). \qed

\vskip4mm

{\bf 4.3.}  {\it Example.} Suppose that  $X$ belongs to the domain of attraction of a stable law of exponent 1, or equivalently, that as $x\to\infty$, 
\beqn\label{Ex}
1-F(x)\sim pL(x)/x \quad \mbox{and}\quad  F(-x) \sim (1-p)L(x)/x 
\eeqn
 for some constant   $0\leq p\leq 1$  and s.v. function  $L$.
 Then (Ha) holds if and only if  
\beqn\label{Sp.c} \rho_n :=  P[S_n >0] \to  1
\eeqn
(cf.  \cite{KM/lmp}). 
Here $S_n$ is a random walk with  $S_0=0$ and the step  distribution given   by $F$.   (\ref{Sp.c})---hence (Ha)---holds whenever $p>1/2$ and  for a (small but significant) sub-class of $F$ with $p=1/2$. 
Suppose $E|X|= \infty$. 
Then 
for $p>1/2$, condition  (Hac)  is satisfied with $\kappa =(2p-1)$ (see Remark \ref{rem4}) so that  
Corollary \ref{cor1}  yields that
\[ 
\frac{U(x, x+h]}{h}\; \sim \frac1{(2p-1)^{2}\ell(x)} \times \left\{ \begin{array}{ll} p \quad &\mbox{as}\;\; x\to\infty, \\
1-p  \quad  & \mbox{as}\;\; x\to -\infty. 
\end{array}\right.
\]
This formula is gotten in 
  \cite[Theorem 3.6]{Ber} (where $F$ is assumed to be arithmetic of span one) under some auxiliary restriction imposed on $P[X=y]$.  
 
  Here we derive a criterion for  transience: if (\ref{Ex}) holds then $F$ is transient if and only if
\beqn\label{Trans}
\int_{1}^\infty \frac{H(t)}{\big(L(t)\vee |A(t)|\big)^2}dt <\infty.
\eeqn
[If $E|X|<\infty$ in addition, (\ref{Trans}) is equivalent to  $EX\neq 0$ as is directly checked.]
In a manner similar to that  leading  to (\ref{<})  one deduces that for $\theta>0$ small
\beqn\label{T/R}
 \theta \big|A(1/\theta)\big|{\bf1}\Big(\pi L(1/\theta) < |A(1/\theta)|\Big)  
\leq \frac{\big|E[\sin \theta X]\big|}{1+o(1)} \leq \theta \big(|A(1/\theta)| + L(1/\theta)  \big), 
\eeqn
while it is known (cf. e.g.,  \cite{P}) that  
$$1- E[\cos \theta X] \sim {\textstyle \frac12} \pi \theta L(1/\theta),$$
 so that $|1-\phi(\theta)| \asymp \theta \big(L(1/\theta) + |A(1/\theta)|\big)$ ($\theta \downarrow 0$). Hence
$$\Re\frac{1}{1-\phi(\theta)} \asymp \frac{ L(1/|\theta|)/|\theta|}{L^2(1/|\theta|)+ A^2(1/|\theta|)}. $$
Thus we have the asserted criterion in view of Ornstein's theorem.  [In the particular case when  $\rho_n \to 0$ or $1$  the criterion (\ref{Trans}) is obtained from Lemma \ref{lem6}.] By (\ref{T/R}) together with the determination of the sequences  of norming  $S_n$ for the convergence to a stable law as given by  \cite[Theorem XVII.5.3]{F} it  is easy to see that $\rho_n\to1$ (resp. 0) if and only if $A(x)/L(x)\to\infty$ (resp $-\infty$) (cf. also \cite{KM/lmp}),  and  
$\rho_n$ remains in a closed interval of $(0, 1)$ if and only if  $|A(x)|=O(L(x))$;  in the latter case   (\ref{Trans}) is reduced to $\int_1^\infty [tL(t)]^{-1}dt <\infty$, and the result is 
obtained in \cite{Wei} by using the Gnedenko local limit theorem.

 For  an arithmetic walk satisfying   (\ref{Ex})
 Berger \cite[Theorem3.5]{Ber} shows that if  $\rho_n$ tends to  $\rho \in (0,1)$, then $U\{n\}/\sum_{k=n}^\infty 1/[kL(k)] \to c(\rho)$ where $c(\rho)$ is a positive function
(expressed explicitly), provided $F$ is transient. The result can be extended to non-arithmetic walks. 
This combined with the consequence of Corollary \ref{cor1} mentioned above shows that if $S_n$ is transient---equivalently (\ref{Trans}) holds---and $\lim \rho_n = \rho$, then as $x\to\infty$
  \[
 \frac{U[x,x+h)}{h} \sim \left\{\begin{array} {ll}
|2p-1|^{-2}p \big/ \ell(x) \quad &\mbox{if} \quad p\neq 1/2, \\[2mm]
 {\displaystyle 1/[2m(x)]}\quad &\mbox{if} \quad p=1/2, \, \rho\in \{0,1\},\\[2mm]
 {\displaystyle c(\rho) \int_x^\infty \frac{dt}{t L(t)} } \quad &\mbox{if} \quad  p=1/2, \,  0< \rho <1.
\end{array}\right.
\]
Here $h$ is understood to be a positive multiple of the span of $F$ if $F$ is arithmetic.

Finally  we present   two particular  cases of (\ref{Ex})  that verify incomparability of (Ha) and (Hb). Let $p=1/2$ in (\ref{Ex}). If
$L(x) \sim (\log x)^2$, we may have $K(x)\sim (\log x)^{\de}/x$ with $0<\de<2$  so  that $A(x)\sim (\log x)^{\de +1}/(\de +1)$, showing that if $1/2< \de \leq 1$,  (Hb) holds  whereas (Ha) fails.    
Similarly  if $L(x) \sim 1$ and  $K(x)\sim (\log x)^{-\de}/x$ with $1/2<\de<1$, then  (Ha) holds  but (Hb) does not.

\end{document}